\documentclass[12pt]{article}
\usepackage{amsmath,amsthm,amssymb}
\usepackage{graphicx}

\newcommand{\C}{\mbox{\rm \,l\kern-0.52em C}}
\newcommand{\Ce}{\rm \,l\kern-0.35em C}

\newtheorem{theorem}{Theorem}[section]

\renewenvironment{proof}{{\bf Proof:}}{\mbox{}\hfill $\Box$}

\theoremstyle{definition}
\date{}

\title{A remark about a theorem of Skandalis}

\author{Michael Puschnigg}
\date{}

\begin{document}

\maketitle 

In 1988 Georges Skandalis proved the remarkable

\begin{theorem} \cite{Sk}
Let $\Gamma$ be a word-hyperbolic group with Kazhdan's property $(T)$ and let $C^*_r(\Gamma)$ be its reduced group $C^*$-algebra. Then the canonical homomorphism
$$
C^*_r(\Gamma)\otimes_{max}C^*_r(\Gamma)\,\longrightarrow\,C^*_r(\Gamma)\otimes_{min}C^*_r(\Gamma)
\eqno(1,1)
$$
from the maximal to the minimal tensor square of $C^*_r(\Gamma)$ does not induce an isomorphism of topological $K$-groups:
$$
K_*(C^*_r(\Gamma)\otimes_{max}C^*_r(\Gamma))\,\overset{\neq}{\longrightarrow}\,K_*(C^*_r(\Gamma)\otimes_{min}C^*_r(\Gamma)).
\eqno(1.2)
$$
\end{theorem}

He deduced that Kasparov's $\gamma$-element \cite{Ka} in equivariant bivariant $K$-theory is different from the unit element for such groups: 
$$
\gamma\neq 1\,\in\,KK^\Gamma({\mathbb C},{\mathbb C}).
\eqno(1.3)
$$
Skandalis theorem left unanswered the question whether it is the injectivity, the surjectivity, or both that fail for the map (1.2). We observe  

\begin{theorem}
Let $\Gamma$ be a word-hyperbolic group. Then the homomorphism (1.2) is surjective.
\end{theorem}

\begin{proof}
Let $\Gamma$ be a word-hyperbolic group.
The full and reduced assembly maps \cite{Ka} with coefficients in the $C^*$-algebra $C^*_r(\Gamma)$, equipped with the trivial $\Gamma$-action, yield the commutative diagram
$$
\begin{array}{ccc}
K_*^{\Gamma}(\underline{E\Gamma},C^*_r(\Gamma)) & 
\overset{\mu_{max}}{\longrightarrow} & K_*(C^*_{max}(\Gamma)
\otimes_{max}C^*_r(\Gamma)) \\
& & \downarrow \\
\parallel & & K_*(C^*_r(\Gamma)\otimes_{max}C^*_r(\Gamma)) \\
& & \downarrow \\
K_*^{\Gamma}(\underline{E\Gamma},C^*_r(\Gamma)) & \overset{\mu_{red}}{\longrightarrow} & K_*(C^*_r(\Gamma)\otimes_{min}C^*_r(\Gamma)) \\
\end{array}
\eqno(1.4)
$$
According to Lafforgue \cite{La}, the Baum-Connes conjecture with coefficients holds for word-hyperbolic groups, so that the lower horizontal map in the previous diagram is an isomorphism. It follows that the lower vertical map on the 
right hand side of the diagram is surjective, which is the content of theorem (0.2).
\end{proof}

\newpage

Having a closer look at the proof of Skandalis' theorem we observe

\begin{theorem}
Let $\Gamma$ be a word-hyperbolic group with Kazhdans Property $(T)$ and let $p_{Kaz}\in C^*_{max}(\Gamma)$ be the Kazhdan projection. (It is characterized by the fact that for a unitary representation $(\pi,{\mathcal H})$ of $\Gamma$ the operator $\pi(p_{Kaz})$ equals the orthogonal projection onto the $\pi(\Gamma)$-fixed vectors in $\mathcal H$.) Let $\Delta:C^*_{max}(\Gamma)\,\to\,C^*_r(\Gamma)\otimes_{max}C^*_r(\Gamma)$ be the homomorphism of $C^*$-algebras induced 
by the diagonal map $\Gamma\to\Gamma\times\Gamma$. Then
$$
[\Delta(p_{Kaz})]\neq 0\,\in\,K_0(C^*_r(\Gamma)\otimes_{max} C^*_r(\Gamma))\otimes_{\mathbb Z}{\mathbb Q}.
\eqno(1.5)
$$
Moreover this element lies in the kernel of the homomorphism (1.2).
\end{theorem}

\begin{proof}
We recall Skandalis' proof of theorem (0.1) \cite{Sk}. He establishes a commutative diagram of $C^*$-algebras
$$
\begin{array}{cccccccccc}
0 & \to & J & \overset{j}{\to} & 
C^*_r(\Gamma)\otimes_{max}C^*_r(\Gamma) & 
& \to & C^*_r(\Gamma)\otimes_{min}C^*_r(\Gamma) & \to & 0\\
&&&&&&&&& \\
&& \downarrow && \downarrow &&&&& \\
&&&&&&&&&\\
&&{\mathcal K}(\ell^2(\Gamma)) & \to & {\mathcal L}(\ell^2(\Gamma)) &&&&& \\
\end{array}
\eqno(1.6)
$$
with right vertical arrow given by the biregular representation of $\Gamma$ and with exact upper line.
The image of the projection $\Delta(p_{Kaz})$ in $C^*_r(\Gamma)\otimes_{min}C^*_r(\Gamma)=C^*_r(\Gamma\otimes\Gamma)$ is zero (the regular representation of $\Gamma\times\Gamma$ has no invariant vectors), while its image in ${\mathcal L}(\ell^2(\Gamma))$ is not (the adjoint representation of $\Gamma$ has invariant vectors). So $\Delta(p_{Kaz})\in J$ maps to a non-zero projection in ${\mathcal K}(\ell^2(\Gamma))$. Because the class of such a projection is non-trivial in $K_0({\mathcal K}(\ell^2(\Gamma))\otimes_{\mathbb Z}{\mathbb Q}\,\simeq\,{\mathbb Q}$, Skandalis deduces that the class $[\Delta(p_{Kaz})]\in K_0(J)\otimes_{\mathbb Z}{\mathbb Q}$ cannot be zero. As $j_*:\,K_0(J)\otimes_{\mathbb Z}{\mathbb Q}\,\to\,K_0(C^*_r(\Gamma)\otimes_{max}C^*_r(\Gamma))\otimes_{\mathbb Z}{\mathbb Q}$ is injective by theorem (0.2) the assertion follows.
\end{proof}

\end{document}